\documentclass[12pt]{amsart}
\usepackage[mathscr]{eucal}
\usepackage{amssymb, amsfonts, ytableau, comment, color}
\usepackage{graphicx}
\usepackage{stmaryrd}

\theoremstyle{plain} 
\newtheorem{thm}{Theorem}[section]

\newtheorem{lem}[thm]{Lemma}
\newtheorem{prop}[thm]{Proposition}

\theoremstyle{definition}

\newtheorem{ex}[thm]{Example}

\theoremstyle{remark}

\numberwithin{equation}{section}

\def\NN{{\mathbb N}}

\def\fS{{\mathfrak S}}

\def\cF{\mathcal F}

\def\height{\operatorname{ht}}

\def\Image{\operatorname{Im}}

\def\too{\longrightarrow}

\def\Im{\operatorname{Im}}

\def\Ker{\operatorname{Ker}}

\def\cF{{\mathcal F}}

\def\SYT{\operatorname{SYT}}
\def\Tab{\operatorname{Tab}}

\def\Tor{\operatorname{Tor}}

\def\ISp{I^{\rm Sp}}

\def\too{\longrightarrow}

\def\chara{\operatorname{char}}

\def\<{{\langle}}
\def\>{{\ranogl}}

\def\sgn{\operatorname{sgn}}

\title[{Minimal free resolutions of the Specht ideals of  $(n-2,2), (d,d,1)$}]{Elementary construction of minimal \\free  resolutions of the Specht ideals \\ of shapes $(n-2,2)$ and $(d,d,1)$}

\author{Kosuke Shibata}
\address{Department of Mathematics, Okayama University, Okayama, Okayama 700-8530, Japan}
\email{pfel97d6@s.okayama-u.ac.jp}
\author{Kohji Yanagawa}
\address{Department of Mathematics, Kansai University, Suita, Osaka 564-8680, Japan}
\email{yanagawa@kansai-u.ac.jp}
\thanks{The second author is partially supported by JSPS Grant-in-Aid for Scientific Research (C) 19K03456.}
\date{--, --, --}

\begin{document}
\maketitle
\begin{abstract}
For a partition $\lambda$ of $n \in \NN$, let $I^{\rm Sp}_\lambda$ be the ideal of $R=K[x_1,\ldots,x_n]$ generated by all Specht polynomials of shape $\lambda$. We assume that ${\rm char}(K)=0$.
Then $R/I^{\rm Sp}_{(n-2,2)}$ is Gorenstein, and $R/I^{\rm Sp}_{(d,d,1)}$ is a Cohen-Macaulay ring with a linear free resolution.  
In this paper, we construct minimal free resolutions of these rings.  Berkesch Zamaere, Griffeth, and Sam \cite{BSS} had already studied minimal free resolutions of $R/\ISp_{(n-d,d)}$, which are also Cohen-Macaulay, 
using heighly advanced technique of the representation theory. However we only use the basic theory of Specht modules, and explicitly describe the differential maps. 
\end{abstract}

\section{Introduction}
For a positive integer $n$,  a {\it partition} of $n$ is a sequence $\lambda = (\lambda_1, \ldots, \lambda_l)$ of integers with $\lambda_1 \ge \lambda_2 \ge \cdots \ge \lambda_l \ge 1$ and $\sum_{i=1}^l \lambda_i =n$.  
A partition $\lambda$ is frequently represented by its Young diagram.  
The {\it Young tableau} of shape $\lambda$ is a bijective filling of the Young diagram of  $\lambda$
by the integers in $[n]:=\{1,2, \ldots, n\}$. 
 For example, the following is a tableau of shape  $(4,2,1)$. 
\begin{equation}\label{ex of T}
\ytableausetup{mathmode, boxsize=1.3em}
\begin{ytableau}
3 & 5 & 1& 7   \\
6 & 2    \\
4 \\
\end{ytableau}
\end{equation}
Let $\Tab(\lambda)$ be the set of  Young tableaux of shape $\lambda$. 

Let $R=K[x_1, \ldots, x_n]$ be a polynomial ring over a field $K$,  and  consider  a tableau $T \in \Tab(\lambda)$ for  a partition $\lambda$ of $n$.   
If the $j$-th column of $T$ consists of $j_1, j_2, \ldots, j_m$ in the order from top to bottom, then 
$$f_T (j) := \prod_{1 \le s < t \le m} (x_{j_s}-x_{j_t}) \in R.$$
 The {\it Specht polynomial} $f_T$ of $T$ is given by  
$$f_T := \prod_{j=1}^{\lambda_1} f_T(j).$$
For example, if $T$ is the tableau \eqref{ex of T}, then $f_T=(x_3-x_6)(x_3-x_4)(x_6-x_4)(x_5-x_2).$

The symmetric group $\fS_n$ acts on the vector space spanned by 
$\{ \, f_T  \mid T \in \Tab(\lambda) \}$. 
 As an $\fS_n$-module, this vector space is isomorphic to the  {\it Specht module} $V_\lambda$, which is very important in the representation  theory of  symmetric groups (cf. \cite{Sa}). Here we study the Specht {\it ideal} 
$$\ISp_\lambda :=(\, f_T  \mid T \in \Tab(\lambda) )$$
of $R$. In the previous paper \cite{Y}, the second author showed the following.  

\begin{thm}[{\cite[Proposition~2.8 and Corollary~4.4]{Y}}]
\label{prev paper main}
If $R/\ISp_\lambda$ is Cohen--Macaulay, then one of the following conditions holds.  
\begin{itemize}
\item[(1)] $\lambda =(n-d, 1, \ldots, 1)$, 
\item[(2)] $\lambda =(n-d,d)$, 
\item[(3)] $\lambda =(d,d,1)$. 
\end{itemize}
If $\chara(K)=0$, the converse is also true. 
\end{thm}

The case (1) is treated in the joint paper \cite{WY} with J. Watanabe, and it is shown that a  minimal free resolution of $R/\ISp_{(n-d,1,\ldots,1)}$ is given  by the Eagon-Northcott complex of a ``Vandermonde-like" matrix.  Since  $\ISp_{(n-1,1)}$ is a linear complete intersection, its free resolution is easy. 
The second author showed that $R/\ISp_{(n-2,2)}$ is a 2-dimensional Gorenstein ring (\cite[Proposition~5.2]{Y}).  
Similarly, when $\chara(K)=0$, $R/\ISp_{(d,d,1)}$ is Cohen--Macaulay and has a linear free resolution. In the present paper, we will construct minimal free resolutions of $R/\ISp_{(n-2,2)}$ and  $R/\ISp_{(d,d,1)}$.   

However, after an earlier version was submitted, the authors were  informed that minimal free resolutions of  
$R/\ISp_{(n-d,d)}$ for $1 \le d \le n/2$ had been studied by Berkesch Zamaere, Griffeth, and Sam  
 \cite{BSS}. More precisely, \cite{BSS} determined the $\fS_n$-module structure of $\Tor_i^R (K, R/\ISp_{(n-d,d)})$.  
(They called $\ISp_{(n-d,d)}$ the ``$(d+1)$-equal ideal''. Of course, this name comes from the decomposition \eqref{decomposition} below.)
However the paper \cite{BSS} dose not give the differential maps of their resolutions, and uses  highly  advanced tools of the representation theory (rational Cherednik algebras, Jack polynomials, etc).  
By contrast, the present paper describes the differential maps explicitly, and uses only the basic theory of Specht modules.   
Here we  do not  use results of \cite{BSS}  to make the exposition self-contained.

It is also noteworthy that, recently, many people study monomial ideals in $R$ on which the symmetric group $\fS_n$ naturally acts (cf. \cite{B+,MR}). 
However, their behavior is quite different from that of Specht ideals. For example, $\Tor_i^R (K, R/\ISp_{(n-2,2)})$ and  $\Tor_i^R (K, R/\ISp_{(d,d,1)})$ are irreducibe as $\fS_n$-modules (the same holds for $\ISp_{(n-d,d)}$ as shown in \cite{BSS}),  but this is far from true for symmetric monomial ideals (cf. \cite{MR}).


\section{Preliminaries and backgrounds}
In this section, we briefly explain Specht modules and related notions. See \cite[Chapter 2]{Sa} for details. For a partition $\lambda=(\lambda_1,\ldots , \lambda_l)$, we sometimes use ``exponential notation''. For example,  $(4,3^2,2,1^3)$ means $(4,3,3,2,1,1,1)$.  If $\lambda = (\lambda_1, \ldots, \lambda_l)$, then $\Tab(\lambda)$ can be simply written as $\Tab(\lambda_1, \ldots, \lambda_l)$. 

We say a tableau $T$ is {\it standard}, if all columns (resp. rows) are increasing from top to bottom (resp. from left to right).
Let $\SYT(\lambda)$ be the set of standard Young tableaux of shape $\lambda$.

Given any set $A$, let $S(A)$ be the set of all permutations of $A$.
Suppose that $T\in \Tab(\lambda)$ has columns $C_1,\ldots ,C_k$. Then 
$C(T):=S(C_1)\times \cdots \times S(C_k)$
is the {\it column-stabilizer} of $T$.

For $T, T'\in \Tab(\lambda)$, $T$ and $T'$ are {\it row equivalent}, if corresponding rows of $T$ and $T'$ contain the same elements. For $T\in \Tab(\lambda)$, the {\it tabloid} $\mbox{\boldmath $\{$}T\mbox{\boldmath $\}$}$ of $T$ is defined by 
$$\mbox{\boldmath $\{$}T\mbox{\boldmath $\}$}:=\{T'\in \Tab(\lambda)\, |\, T \, \, {\rm and} \, \, T' \, \, {\rm are\, \,  row \, \, equivalent}\},$$
 and the {\it polytabloid} of $T$ is defined by
$$e(T):=\displaystyle \sum_{\pi\in C(T)}\sgn(\pi) \, \pi\mbox{\boldmath $\{$}T\mbox{\boldmath $\}$}.$$
It is easy to see that $e(T)=\sgn(\sigma) \, e(\sigma T)$ for $\sigma \in C(T)$. 

The vector space 
$$V_\lambda:=\displaystyle \sum_{T\in \Tab(\lambda)}Ke(T)$$
becomes an $\mathfrak{S}_n$-module in the natural way, and it is called the {\it Specht module} of $\lambda$.
If $\chara(K)=0$, the Specht modules $V_\lambda$ are irreducible, and $V_\lambda$ for partitions $\lambda$ of $n$ form a complete list of irreducible representations of $\mathfrak{S}_n$.

In the previous section, we defined the Specht polynomial $f_T \in R=K[x_1, \ldots, x_n]$. 
Since $\mathfrak{S}_n$ acts on $R$, the vector subspace
$$\sum_{T\in \Tab(\lambda)} Kf_T$$
is also an $\mathfrak{S}_n$-module.
Moreover, the map
\begin{equation}\label{Isom}
V_{\lambda}\xrightarrow{\cong} \sum_{T\in \Tab(\lambda)} Kf_T, \, \, \, \, e(T) \longmapsto f_T .
\end{equation}
is well-defined, and gives an isomorphism as $\mathfrak{S}_n$-modules.

Note that $\{e(T)\, |\, T\in \Tab(\lambda)\}$ is linearly dependent, and there are relations called {\it Garnir relations}. Its definition for general $\lambda$ becomes long, so we explain it using our examples. See \cite[\S 2.6]{Sa} for the general case.

For
\begin{equation}\label{T for Gor}
T= 
\ytableausetup
{mathmode, boxsize=3em}
\begin{ytableau}
a_1 & b_1 & c_1& c_2 & \cdots & c_{n-3-i}   \\
a_2 & b_2    \\
\vdots \\
a_{i+1}
\end{ytableau}
\in \Tab(n-1-i,2, 1^{i-1})
\end{equation}
and $A=\{a_1,\ldots, a_{i+1}\}$, $B=\{b_1\}$, set
\begin{equation*}
S_T(A,B):=\left\{ \sigma \in \mathfrak{S}_n \, \middle| \,  \; \parbox[center]{13.3em}{$\sigma(i)=i$ for $i\notin A\cup B$,\\
$\sigma(a_1)<\sigma(a_2)<\cdots<\sigma(a_{i+1})$.} \right\}
\end{equation*}
(if there is no danger of confusion, we write $S(A,B)$ for $S_T(A,B)$) 
and 
\begin{equation}\label{G element}
g_{A,B} :=\sum_{\sigma\in S(A,B)} \sgn(\sigma) \, \sigma.
\end{equation}
Here we regard $g_{A,B}$ as an element of the group ring of $\fS_n$.  
Then we have
\begin{equation}\label{G relation}
g_{A,B}e(T)=\sum_{\sigma\in S(A,B)} \sgn(\sigma)e(\sigma T)=0
\end{equation}
by \cite[Proposition~2.6.3]{Sa}.

Next, for $A=\{a_2,\ldots, a_{i+1}\}$, $B=\{b_1,b_2\}$, set
\begin{equation*}
S(A,B):=\left\{ \sigma \in \mathfrak{S}_n \, \middle| \,  \; \parbox{20em}{\begin{center}$\sigma(i)=i$ for $i\notin A\cup B$, \\
 $\sigma(a_2)<\sigma(a_3)<\cdots<\sigma(a_{i+1})$, $\sigma(b_1)<\sigma(b_2)$ \end{center}} \right\}. 
\end{equation*}
Then $g_{A,B}$ is given in the same way as \eqref{G element}, and  we have $g_{A,B}e(T)=0$ as in \eqref{G relation} again. The same is true for  $A=\{b_1, b_2\}, B=\{c_1\}$, and $A=\{c_i \}$, $B=\{c_{i+1}\}$. 
In these cases, $g_{A,B}$ is called the {\it Garnir element} associated with $A$ and $B$. 
 It is a classical result that $\{e(T)\mid T\in \SYT(\lambda)\}$ is a basis of $V_\lambda$ (cf. \cite[Theorem 2.6.5]{Sa}), and that $\{e(T)\mid T\in \SYT(\lambda)\}$ spans $V_\lambda$ is shown using  Garnir relations. 

\begin{ex}\label{Garnir example}
For
$$
T= 
\ytableausetup
{mathmode, boxsize=1.3em}
\begin{ytableau}
2 &1 & 6\\
3& 5    \\
4
\end{ytableau},
$$
set $A=\{2,3,4\}$ and $B=\{1\}$, then
$$
g_{A,B}e(T)=e(T)
-e(
\ytableausetup
{mathmode, boxsize=1.3em}
\begin{ytableau}
1 &2 & 6\\
3& 5    \\
4
\end{ytableau})
+e(
\begin{ytableau}
1 &3 & 6\\
2& 5    \\
4
\end{ytableau})
-e(
\begin{ytableau}
1 &4 & 6\\
2& 5    \\
3
\end{ytableau})\\
=0.
$$
Next we consider the following tableau
\begin{equation}\label{T example} 
T= 
\ytableausetup
{mathmode, boxsize=1.3em}
\begin{ytableau}
1 &2 & 6\\
4& 3    \\
5
\end{ytableau}.
\end{equation}
Set $A=\{4,5\}$ and $B=\{2,3\}$, then
\begin{eqnarray*}
g_{A,B}e(T)&=&e(T)
-e(
\ytableausetup
{mathmode, boxsize=1.3em}
\begin{ytableau}
1 &2 & 6\\
3& 4    \\
5
\end{ytableau})
+
e(\begin{ytableau}
1 &2 & 6\\
3& 5    \\
4
\end{ytableau})
-e(
\begin{ytableau}
1 &3 & 6\\
2& 5    \\
4
\end{ytableau})\\
&&+e(
\begin{ytableau}
1 &3 & 6\\
2& 4    \\
5
\end{ytableau})
+e(
\begin{ytableau}
1 &4 & 6\\
2& 5    \\
3
\end{ytableau})\\
&=&0.
\end{eqnarray*}
\end{ex}

\begin{prop}[cf. {\cite[Theorem 2.6.4]{Sa}}]\label{relation generated}
Any linear relations among $\{e(T)\mid T\in \Tab(\lambda)\}$ is a linear combination of Garnir relations.
That is, if 
\begin{equation}\label{linear relation}
\displaystyle \sum_{i=1}^m a_i e(T_i)=0
\end{equation}
in $V_\lambda$ for $T_1,\ldots, T_m\in \Tab(\lambda)$ and $a_i\in K$, then $\sum_{i=1}^m a_i T_i$ (this is a formal sum, and there is no relation among $T_1,\ldots, T_m$) is contained in the linear space $V$ spanned by 
$$
\left \{  \sum_{\sigma\in S(A,B)} \sgn(\sigma) \, \sigma T \ \middle| \   \text{$S(A,B)$ gives a Garnir element $g_{A,B}$} \right \}. 
$$
\end{prop}

\begin{proof}
Assume that \eqref{linear relation} holds. Each $e(T_i)$ can be rewritten as 
$$
e(T_i)=\sum_{T \in \SYT(\lambda)} b_{i,T}e(T)
$$
for some $b_{i,T} \in K$ using only  Garnir relations, see the proof of {\cite[Theorem 2.6.4]{Sa}}. Hence we have 
$$v_i:= T_i - \sum_{T \in \SYT(\lambda)} b_{i,T} T \in V$$
for each $i$. Note that 
$$\sum_{i=1}^m \sum_{T \in \SYT(\lambda)} a_i b_{i,T}e(T) =\sum_{i=1}^m a_i e(T_i)=0.$$
However, since  $\{e(T)\mid T\in \SYT(\lambda)\}$ is a basis, we have  $\sum_{i=1}^m a_i b_{i,T}=0$ for all $T \in \SYT(\lambda)$. Hence 
$$ \sum_{i=1}^m a_i T_i  = \sum_{i=1}^m a_i v_i \in V.$$
\end{proof}

In the rest of this section, we collect a few remarks on the Specht ideals  $\ISp_{(n-d,d)}$ and $\ISp_{(d,d,1)}$. 
First, we have the decomposition 
\begin{equation}\label{decomposition}
\ISp_{(n-d,d)}=\bigcap_{\substack{F \subset [n] \\ \# F = d +1}}(x_i -x_j \mid i, j \in F), 
\end{equation}
and the same is true for $\ISp_{(d,d,1)}$. So these ideals can be seen as   special cases of the ideals associated with  subspace arrangements (cf.  \cite{BCES,LL}). The second author (\cite{Y}) made much effort to show that $\sqrt{\ISp_{\lambda}}=\ISp_{\lambda}$ 
for $\lambda=(n-d,d), (d,d,1)$, but it directly follows from \cite[Corollary~3.2]{LL}.  
 
To prove the Cohen--Macaulay-ness of $\ISp_{(n-d,d)}$ and $\ISp_{(d,d,1)}$ in characteristic 0, the second author (\cite{Y}) cited  a result in \cite{EGL}, which uses the representation theory of rational Cherednik algebras. 
Recently, McDaniel and Watanabe (\cite{MW}) gave a purely ring theoretic proof. 
Moreover, in the positive characteristic case, they showed that   $R/\ISp_{(n-d,d)}$  (resp. $R/\ISp_{(d,d,1)}$) is Cohen--Macaulay if and only if $\chara(K) \ge d$ (resp. $\chara(K) \ge d+1$). 

As stated in \cite{BSS, SY}, 
$\ISp_{(d,d,1)}$ has a linear free resolution if it is Cohen--Macaulay. Anyway, this is an easy consequence of \cite[Theorem 5.3.7]{V}.

\section{The case $(n-2,2)$: construction}
For  $R/\ISp_{(n-2,2)}$, we define the chain complex 
\begin{equation}\label{Gor}
\cF_\bullet^{(n-2,2)}:0 \too F_{n-2} \stackrel{\partial_{n-2}}{\too}  F_{n-3} \stackrel{\partial_{n-3}}{\too}  \cdots \stackrel{\partial_2}{\too}  F_1 \stackrel{\partial_1}{\too} F_0 \too 0
\end{equation}
 of graded free $R$-modules as follows. 
Here  $F_0 =R$, $F_1 = V_{(n-2,2)} \otimes_K R(-2)$, 
$$
F_i = V_{(n-1-i,2, 1^{i-1})} \otimes_K R(-1-i)
$$
for $1 \le i \le n-3$, and $F_{n-2}= V_{(1^n)} \otimes_K R(-n)$. 
For $T \in \Tab(n-2,2)$, set $\partial_1(e(T)\otimes 1) :=f_T \in R =F_0$.  To describe $\partial_i$ for $ 2
 \le i \le n-3$, we need the preparation. For  the tableau $T \in \Tab(n-1-i,2, 1^{i-1})$ of \eqref{T for Gor} 
and $j$ with $1 \le j \le i+1$, set 
$$
T_j:= 
\ytableausetup
{mathmode, boxsize=3em}
\begin{ytableau}
a_1 & b_1 & c_1& c_2 & \cdots & c_{n-3-i}  & a_j \\
a_2 & b_2    \\
\vdots \\
a_{j-1}\\
a_{j+1} \\
\vdots \\
a_{i+1}
\end{ytableau}
\in \Tab(n-i,2, 1^{i-2}).
$$
Then we set
 $$\partial_i(e(T) \otimes 1):=\sum_{j=1}^{i+1} (-1)^{j-1} e(T_j) \otimes x_{a_j} \in V_{(n-i,2, 1^{i-2})} \otimes_K R(-i)=F_{i-1}.$$

Recall that $e(\sigma T)=\sgn(\sigma)e(T)$ for $\sigma \in C(T)$. It is easy to check that the construction of $\partial_i$  is compatible with this principle, that is, $\partial_i(e(\sigma T)\otimes 1)=\sgn(\sigma) \, \partial_i(e(T)\otimes 1)$ holds for $\sigma \in C(T)$. However, this is not enough.  
Since $\{e(T)\, |\, T\in \Tab(\lambda)\}$ is linearly dependent, the well-definedness of $\partial_i$ is still non-trivial. We will show this in Theorem~\ref{wdGor} below.

Finally, we define the differential map 
$$\partial_{n-2}: V_{(1^n)} \otimes_K R(-n) \too 
V_{(2,2,1^{n-4})} \otimes_K R(-n+2).$$
Since $\dim V_{(1^n)}=1$, it suffices to define $\partial_{n-2}(e(T))$ for  
\begin{equation}\label{T for Gor last}
T:= 
\ytableausetup{mathmode, boxsize=1.3em}
\begin{ytableau}
1 \\
2  \\
\vdots \\
n
\end{ytableau}
\in \Tab(1^n),
\end{equation}
and we do not have to care the well-definedness. For $j,k$ with $1 \le j < k \le n$, set 
$$
T_{j,k} := 
\ytableausetup{mathmode, boxsize=1.3em}
\begin{ytableau}
\vdots & j\\
\vdots  & k\\
\vdots \\
\vdots 
\end{ytableau}
\in \Tab(2,2, 1^{n-4}),
$$
where the first column is the ``transpose'' of $$\ytableausetup
{boxsize=2.5em}\begin{ytableau}
1 & 2 & \cdots & j-1 & j+1 & \cdots  & k-1 & k+1 &\cdots & n \end{ytableau}.$$

Then $$\partial_{n-2}(e(T) \otimes 1):=\sum_{1 \le j < k \le n}(-1)^{j+k-1} e(T_{j,k}) \otimes x_{j}x_{k} \in V_{(2,2, 1^{n-4})} \otimes_K R(-n+2) \in F_{n-3}.$$

We define the $\mathfrak{S}_n$-module structure on $F_i=V_\lambda \otimes_K R(-j)$ (here $\lambda$ is a suitable partition of  $n$ and $j$ is a suitable integer) by $\sigma(v\otimes f):=\sigma v \otimes \sigma f$ for $\sigma \in\fS_n$. 
By (\ref{Isom}), $\partial_1$ is an $\mathfrak{S}_n$-homomorphism. In fact, we have 
$$\partial_1(\sigma(e(T) \otimes g)) = \partial_1(\sigma(e(T)) \otimes \sigma g) = \sigma(f_T)  \cdot \sigma g  =\sigma(f_T \cdot g)=\sigma(\partial_1(e(T) \otimes g) )$$
for $T \in \Tab(n-2,2)$ and $g \in R$. 
For $i$ with $2\leq i \leq n-3$, $T\in \Tab(n-i,2, 1^{i-2})$ and $\sigma\in \mathfrak{S}_n$, we have $\sigma(T)_j=\sigma(T_j)$. Hence $\partial_i(\sigma(e(T)\otimes g))=\sigma (\partial_i(e(T)\otimes g))$, that is,  $\partial_i$ are $\mathfrak{S}_n$-homomorphisms. Similarly, $\partial_{n-2}$ is also.

\begin{ex}\label{(4,2)} 
Our minimal free resolution $\cF_\bullet^{(4,2)}$ of $R/\ISp_{(4,2)}$ is of the form  
 $$0 \too V_{\ytableausetup{boxsize=0.25em} 
\begin{ytableau}
\\
\\
\\
\\
\\
\\
\end{ytableau}}\otimes_K R(-6) \stackrel{\partial_{4}}{\too} 
V_{\ytableausetup{boxsize=0.25em} 
\begin{ytableau}
{} & \\
 {} & \\
\\
\\
\end{ytableau}} \otimes_K R(-4) \stackrel{\partial_{3}}
{\too} 
V_{\ytableausetup{boxsize=0.25em} 
\begin{ytableau}
{} & {} & \\
 {} & \\
\\
\end{ytableau}} \otimes_K R(-3) \qquad \qquad  \qquad \qquad \quad $$
\begin{equation}\label{F^(4,2)} \qquad \qquad \qquad \qquad \qquad \qquad \qquad 
 \stackrel{\partial_{2}}{\too} 
V_{\ytableausetup{boxsize=0.25em} 
\begin{ytableau}
{} & {} & {} & \\
 {} & \\
\end{ytableau}} \otimes_K R(-2) \stackrel{\partial_{1}}{\too} 
R \too 0. 
\end{equation}
The differential maps are given by   
\begin{eqnarray*}
\partial_4 ( e ( \, 
\ytableausetup
{mathmode, boxsize=1em}
\begin{ytableau}
1 \\
2    \\
3 \\
4 \\
5 \\
6
\end{ytableau}
\, )  \otimes 1) 
&=& 
 e ( \, 
\ytableausetup
{mathmode, boxsize=1em}
\begin{ytableau}
3  & 1\\
4  & 2\\
5 \\
6
\end{ytableau}
\, )  \otimes x_1x_2 
-   e ( \, 
\ytableausetup
{mathmode, boxsize=1em}
\begin{ytableau}
2 & 1\\
4  & 3\\
5  \\
6
\end{ytableau}
\, )  \otimes x_1x_3 
+  e ( \, 
\ytableausetup
{mathmode, boxsize=1em}
\begin{ytableau}
2  & 1\\
3  & 4\\
5   \\
6 
\end{ytableau}
\, )  \otimes x_1x_4\\
& & - e ( \, 
\ytableausetup
{mathmode, boxsize=1em}
\begin{ytableau}
2 & 1\\
3  & 5\\
4   \\
6
\end{ytableau}
\, )  \otimes x_1x_5
+ e ( \, 
\ytableausetup
{mathmode, boxsize=1em}
\begin{ytableau}
2 & 1\\
3  & 6\\
4   \\
5
\end{ytableau}
\, )  \otimes x_1x_6
+  e ( \, 
\ytableausetup
{mathmode, boxsize=1em}
\begin{ytableau}
1  & 2\\
4  & 3\\
5   \\
6
\end{ytableau}
\, )  \otimes x_2x_3   \\
& &
- e ( \, 
\ytableausetup
{mathmode, boxsize=1em}
\begin{ytableau}
1  & 2\\
3  & 4\\
5  \\ 
6
\end{ytableau}
\, )  \otimes x_2x_4       
+ e ( \, 
\ytableausetup
{mathmode, boxsize=1em}
\begin{ytableau}
1  & 2\\
3  & 5\\
4  \\
6
\end{ytableau}
\, )  \otimes x_2x_5
- e ( \, 
\ytableausetup
{mathmode, boxsize=1em}
\begin{ytableau}
1  & 2\\
3  & 6\\
4  \\
5
\end{ytableau}
\, )  \otimes x_2x_6 \\
&& 
+ e ( \, 
\ytableausetup
{mathmode, boxsize=1em}
\begin{ytableau}
1  & 3\\
2  & 4\\
5  \\
6
\end{ytableau}
\, )  \otimes x_3x_4 
- e ( \, 
\ytableausetup
{mathmode, boxsize=1em}
\begin{ytableau}
1  & 3\\
2  & 5\\
4  \\ 
6
\end{ytableau}
\, )  \otimes x_3x_5 
+ e ( \, 
\ytableausetup
{mathmode, boxsize=1em}
\begin{ytableau}
1  & 3\\
2  & 6\\
4  \\ 
5
\end{ytableau}
\, )  \otimes x_3x_6 
\\
& & 
+ e ( \, 
\ytableausetup
{mathmode, boxsize=1em}
\begin{ytableau}
1  & 4\\
2  & 5\\
3 \\
6
\end{ytableau}
\, )  \otimes x_4x_5
- e ( \, 
\ytableausetup
{mathmode, boxsize=1em}
\begin{ytableau}
1  & 4\\
2  & 6\\
3 \\
5
\end{ytableau}
\, )  \otimes x_4x_6
+ e ( \, 
\ytableausetup
{mathmode, boxsize=1em}
\begin{ytableau}
1  & 5 \\
2  & 6\\
3 \\
4
\end{ytableau}
\, )  \otimes x_5x_6,
\end{eqnarray*}
$$\partial_3(e ( \, 
\ytableausetup
{mathmode, boxsize=1em}
\begin{ytableau}
3  & 1\\
4  & 2\\
5 \\
6
\end{ytableau}
\, )  \otimes 1) = 
e ( \, 
\ytableausetup
{mathmode, boxsize=1em}
\begin{ytableau}
4  & 1 & 3\\
5  & 2\\
6 
\end{ytableau})
\otimes x_3
- e ( \, 
\ytableausetup
{mathmode, boxsize=1em}
\begin{ytableau}
3  & 1 & 4\\
5  & 2\\
6 
\end{ytableau})
\otimes x_4 
+ e ( \, 
\ytableausetup
{mathmode, boxsize=1em}
\begin{ytableau}
3  & 1 & 5\\
4  & 2\\
6 
\end{ytableau})
\otimes x_5 
- e ( \, 
\ytableausetup
{mathmode, boxsize=1em}
\begin{ytableau}
3  & 1 & 6\\
4  & 2\\
5 
\end{ytableau})
\otimes x_6,$$
and 
$$
\partial_2(e ( \, 
\ytableausetup
{mathmode, boxsize=1em}
\begin{ytableau}
4  & 1 & 3\\
5  & 2\\
6 
\end{ytableau})
\otimes 1) = e ( \, 
\ytableausetup
{mathmode, boxsize=1em}
\begin{ytableau}
5  & 1 & 3 & 4\\
6  & 2\\
 \end{ytableau})
\otimes x_4 - 
e ( \, 
\ytableausetup
{mathmode, boxsize=1em}
\begin{ytableau}
4  & 1 & 3 & 5\\
6  & 2\\
 \end{ytableau})
\otimes x_5
+ e ( \, 
\ytableausetup
{mathmode, boxsize=1em}
\begin{ytableau}
4  & 1 & 3 & 6\\
5  & 2\\
 \end{ytableau})
\otimes x_6. 
$$
\end{ex}

\begin{thm}\label{1st thm}
If $\chara(K)=0$, the complex $\cF_\bullet^{(n-2,2)}$ of \eqref{Gor} is a minimal free resolution  of $R/\ISp_{(n-2,2)}$.  
\end{thm}

\section{The case $(n-2,2)$: Proof}

\begin{lem}\label{Gor Betti} 
We have 
$$\beta_i(R/\ISp_{(n-2,2)})=\beta_{i,i+1}(R/\ISp_{(n-2,2)})=\dim_K V_{(n-i-1,2,1^{i-1})}$$
 for all $1\leq i\leq n-3$, and 
$$\beta_{n-2}(R/\ISp_{(n-2,2)})=\beta_{n-2, n}(R/\ISp_{(n-2,2)})=1=\dim_K V_{(1^n)}.$$
\end{lem}

\begin{proof}
By the hook formula (\cite[Theorem~3.10.2]{Sa}), for all $i$ with $1\leq i \leq n-3$, we have 
\begin{eqnarray*}
\dim_K V_{(n-i-1,2,1^{i-1})} \!\! 
&=& \!\!  \frac{n!}{(n-1)(n-(i-1)-2)(n-(i-1)-4)!(i-1+2)!(i-1)!  } \\
\!\!  &=& \!\!  \frac{n!}{(n-1)(n-i-1)(n-i-3)!(i+1)!(i-1)!  } \\
\!\!  &=& \!\!  \frac{n!(n-i-2)i}{(i+1)!(n-i-1)!(n-1)}.
\end{eqnarray*}
On the other hand, $R/\ISp_{(n-2,2)}$ is a Gorenstein ring with the Hilbert series 
$$\frac{1+(n-2)t+t^2}{(1-t)^2}$$
by \cite[Proposition~5.2]{Y} (see  its proof for the Hilbert series), and we have 
\begin{eqnarray*}
\beta_i (R/\ISp_{(n-2,2)})
&=&\beta_{i,i+1}  (R/\ISp_{(n-2,2)})\\
&=&
\binom{n-1}{i+1}\binom{i}{i-1}
+\binom{n-1}{i}\binom{n-i-2}{1}
-\binom{n-2}{i}\binom{n-2}{1}\\
&=& \frac{(n-1)!i}{(i+1)!(n-i-2)!}+\frac{(n-1)!(n-i-2)}{i!(n-1-i)!}-\frac{(n-2)!(n-2)}{i!(n-2-i)!}\\
&=& \frac{n!i}{(i+1)!(n-i-2)!n}+\frac{n!(n-i-2)}{i!(n-1-i)!n}-\frac{n!(n-2)!}{i!(n-2-i)!n(n-1)}\\
&=&\frac{n!(n-i-2)i}{(i+1)!(n-i-1)!(n-1)}
\end{eqnarray*}
for all $i$ with $1\leq i \leq n-3$.  
Here we use \cite[Proposition~5.3.14]{V} (note that  $\ISp_{(n-2,2)}$ is a Gorenstein ideal generated by quadrics and $\height(\ISp_{(n-2,2)})=n-2$). 
So we get the first equation. 

The second one is easy. 
\end{proof}

\begin{thm}\label{wdGor}
The maps $\partial_i$ ($1 \le i \le n-3$) defined in the previous section 
are well-defined.
\end{thm}

The proof of this lemma is elementary, but (therefore?) long and technical. 
By the purpose of the present paper, we do not skip the details. 
Example~\ref{sigma-tau} below, which explain a few details of the proof, must be helpful for better understanding.

Note that an element $\varphi \in V_\lambda \otimes_K R_1$ is uniquely written as $\sum_{i=1}^n v_i \otimes x_i$ ($v_i \in V_\lambda$). We call $v_i \otimes x_i$ the {\it $x_i$-part} of $\varphi$.  

\begin{proof}
The well-definedness of $\partial_1$ is nothing other than that of the map \eqref{Isom}.  So we assume that $2 \le i \le n-3$. By Proposition~\ref{relation generated}, it suffices to show that 
\begin{equation}\label{quasi Garnir}
\sum_{\sigma\in S_T(A,B)} \sgn(\sigma) \,\partial_i(e(\sigma T)\otimes 1)=0
\end{equation} 
for $T\in \Tab(n-1-i,2,1^{i-1})$. Let $T$ be as in \eqref{T for Gor}. 
Then there are three types of $S_T(A,B)$.

{\it Case 1.} When $A=\{b_1, b_2\}$ and $B=\{c_1\}$, or  $A=\{c_l \}$ and $B=\{c_{l+1}\}$: The left side of \eqref{quasi Garnir} can be decomposed to the sum of the $x_{a_1}$-part, \ldots , 
and the $x_{a_{i+1}}$-part. Since $g_{A,B}$ is  a Garnir element also for $T_j$, the $x_{a_j}$-part of the left side of \eqref{quasi Garnir} is
\begin{eqnarray*}
\displaystyle \sum_{\sigma\in S(A,B)} (-1)^{j-1}\sgn(\sigma) \, e((\sigma T)_j) \otimes x_{a_j}&=&\displaystyle (-1)^{j-1}\sum_{\sigma\in S(A,B)} \sgn(\sigma) \, e(\sigma(T_j)) \otimes x_{a_j}\\
&=&\displaystyle (-1)^{j-1}g_{A,B}  e(T_j) \otimes x_{a_j}=0. 
\end{eqnarray*}
Hence \eqref{quasi Garnir} holds in this case.

{\it Case 2.} When $A=\{a_1,\ldots,a_{i+1}\}$ and $B=\{b_1\}$: The left side of \eqref{quasi Garnir} can be decomposed to the sum of the $x_{a_1}$-part, \ldots, the $x_{a_{i+1}}$-part, and the $x_{b_1}$-part. 
To treat the Garnir relation, we may assume that $b_1<a_1<a_2<\cdots<a_{i+1}$.

Fix an integer $j$ with $1\le j \le i+1$, and set $A_j:=A\backslash \{a_j\}$.
Note that 
\begin{eqnarray*}
S(A,B)=\{ \sigma\in S(A,B) \, | \, \sigma(a_j)=a_{j}\} \! \! \! &\sqcup& \! \! \!
\{\sigma\in S(A,B)\, |\,   \sigma(a_{j+1})= a_{j}\} \\
& \sqcup&   \! \! \!  \{\sigma\in S(A,B)\, |\,  \sigma(b_1)= a_{j}\}
\end{eqnarray*}
For $\sigma\in S(A,B)$, $ \sigma(a_j)=a_{j}$ if and only if $\sigma(b_1)\leq a_{j-1}$. 
Similarly, $\sigma(a_{j+1})= a_{j}$ if and only if $\sigma(b_1)\geq a_{j+1}$.

If  $\sigma(a_j) =a_j$, then $\sigma$ also belongs to $S_{T_j}(A_j,B)$, and we have 
$$\{\sigma \in S(A,B)\mid \sigma(a_j)=a_j\}= \{\sigma \in S_{T_j}(A_j,B)\mid \sigma(b_1)\leq a_{j-1}\},$$ and $\sigma(T_j)=(\sigma T)_{j}$. (For notational simplicity, we will write $S(A_j, B)$ for $S_{T_j}(A_j,B)$ below.)

If  $\sigma(a_{j+1}) =a_j$, then we have $\tau:= (\sigma(a_j) \ a_j) \cdot \sigma \in S(A_j, B)$. 
In fact, 
$$\tau(k) = \begin{cases}
a_j & (k=a_j) \\
\sigma(a_j) & (k=a_{j+1}) \\
\sigma(k) & (k \ne a_j, a_{j+1}), 
\end{cases}$$
and hence $\tau$ only moves elements in $A_j \cup B$, and $\tau(k) < \tau(l)$ for $k, l \in A_j$ with $k<l$. We also have $\tau(b_1)=\sigma(b_1) \ge a_{j+1}$, $\sgn(\tau)=-\sgn(\sigma)$, and 
\begin{equation}\label{tau-sigma}
\tau(T_j)=(\sigma T)_{j+1}. 
\end{equation}
(In Example~\ref{sigma-tau} (1) below, we will check \eqref{tau-sigma} very carefully to get the feeling. 
From the next paragraph, we will leave similar computations to the reader as easy exercises.)  
Moreover, the map 
$$
f : \{ \sigma \in S(A,B) \mid \sigma(a_{j+1}) = a_j\} \too \{ \tau \in S(A_j,B) \mid \tau(b_1) \ge a_{j+1} \}
$$
defined by the above operation is bijective. In fact, the inverse map is given by 
$S(A_j, B) \ni \tau \longmapsto (a_j \ \tau(a_{j+1})) \cdot \tau$.

Hence the $x_{a_j}$-part of the left side of \eqref{quasi Garnir} is 
\begin{eqnarray*}
&& \left (
\sum_{ \substack{\sigma\in S(A,B) \\ \sigma(a_j)=a_{j}}}(-1)^{j-1}\sgn(\sigma)e((\sigma T)_{j})  +\sum_{\substack{\sigma\in S(A,B)\\ \sigma(a_{j+1})= a_{j}} }(-1)^{j}\sgn(\sigma)e((\sigma T)_{j+1})) \right) \otimes x_{a_j}\\
&=& \left( \sum_{ \substack{\sigma \in S(A_j,B) \\ \sigma(b_1)\leq a_{j-1}}}(-1)^{j-1}\sgn(\sigma)e(\sigma (T_j)) +\sum_{\substack{\tau\in S(A_j,B) \\ \tau(b_1)\geq a_{j+1} } }(-(-1)^{j}\sgn(\tau)e(\tau (T_{j})) ) \right) \otimes x_{a_j}\\
&=&(-1)^{j-1}g_{A_j,B}e(T_j)\otimes  x_{a_j} \\
&=& 0.
\end{eqnarray*}
Here the last equality follows from that  $g_{A_j,B}$ is a Garnir element for $T_j$. 

It remains to check the $x_{b_1}$-part of the left side of \eqref{quasi Garnir}. Consider the tableau $T':=  ((a_1\, \, b_1)T)_1$, 
that is,  
$$
T'= 
\ytableausetup
{mathmode, boxsize=2em}
\begin{ytableau}
a_2 & a_1 & c_1& c_2 & \cdots & b_1   \\
a_3 & b_2    \\
\vdots \\
a_{i+1}
\end{ytableau}.
$$
Set $A':=A\backslash \{a_1\}$ and  $B':=\{a_1\}$, and consider the map 
$$
f':\{ \sigma \in S(A,B) \mid \sigma(a_1)=b_1\} \ni \sigma \longmapsto  (\sigma(b_1) \ b_1)\sigma \in S_{T'}(A',B').
$$
Of course, we have to check that $\tau := f'(\sigma)$ actually belongs  $ S_{T'}(A',B')$, but it follows from the fact that 
$$
\tau(k) = \begin{cases}
b_1 & (k=b_1) \\
\sigma(b_1) & (k=a_1) \\
\sigma(k) & (k \ne a_1, b_1). 
\end{cases}
$$
Moreover, $f'$ is bijective.  In fact, the inverse map is given by $ S_{T'}(A',B') \ni \tau \longmapsto (\tau(a_1) \ b_1) \, \tau$. 
 We also remark that $\sgn(\sigma)=-\sgn(f'(\sigma))$ and $(\sigma T)_1=f'(\sigma)T'$. 
Hence the $x_{b_1}$-part of the left side of \eqref{quasi Garnir} is 
$$
\Bigl(\sum_{\substack{\sigma\in S(A,B)\\ \sigma(a_1)=b_1}}\sgn(\sigma)(\sigma T)_{1}  ) \Bigr) \otimes x_{b_1}
=-g_{A',B'}e(T')\otimes x_{b_1}
=0.$$

Below, we will use bijections  similar to $f, f'$ repeatedly. Each time, we will define the bijections explicitly,  but we will not check they work well, and leave them as easy exercises.   

{\it Case 3.}  When $A:=\{a_2,\ldots,a_{i+1}\}$ and $B:=\{b_1, b_2\}$:  
Note that the left side of \eqref{quasi Garnir} can be decomposed to the sum of the $x_{a_1}$-part, \ldots, the $x_{a_{i+1}}$-part, the $x_{b_1}$-part and the $x_{b_2}$-part. 
Fix $j$ with $2 \le j \le i+1$, and  set $A_j:=A\backslash \{a_j\}$.  To treat the Garnir relation, we may assume that 
 $b_1<b_2<a_2<\cdots<a_{i+1}$. 
 
First, we treat the $x_{a_j}$-part for $j \ge 2$.  Set
\begin{eqnarray*}
G_1&:=&\{\sigma\in S(A,B) \mid   a_j<\sigma(b_1) \}, \\
G_2&:=&\{\sigma\in S(A,B) \mid  \sigma(b_1)<a_j<\sigma(b_2)\} \\
G_3&:=&\{\sigma\in S(A,B) \mid   a_j>\sigma(b_2)\}, \\
\end{eqnarray*}
and
\begin{eqnarray*}
G'_1&:=&\{\sigma\in S(A_j,B) \mid  a_j<\sigma(b_1)\}, \\
G'_2&:=&\{\sigma\in S(A_j,B) \mid \sigma(b_1)<a_j<\sigma(b_2)\},\\
G'_3&:=&\{\sigma\in S(A_j,B) \mid  a_j>\sigma(b_2)\},  \\
\end{eqnarray*}
Then we have 
$$\{ \sigma \in S(A,B) \mid \sigma(b_1), \sigma(b_2) \ne a_j \} = 
G_1\sqcup G_2\sqcup G_3$$
and  
$$S(A_j,B)=G'_1 \sqcup G'_2\sqcup G'_3.$$

For $\sigma\in S(A,B)$, $\sigma\in G_1$ if and only if $\sigma(a_{j+2})=a_j$. 
Similarly, if $\sigma \in G_2$ (resp. $\sigma \in G_3$), then  $\sigma(a_{j+1})=a_j$ (resp. $\sigma(a_j)=a_j$). 
Hence we have $G_3=G'_3$. 
Moreover, we have the following bijections 
\begin{eqnarray*}
&& f_1: G_1 \ni \sigma \longmapsto (a_j \, \sigma(a_{j+1})\, \sigma(a_{j}))\sigma \in G'_1\\
&& f_2: G_2 \ni \sigma \longmapsto  (a_j \, \sigma(a_{j}))\sigma \in G'_2. \\
\end{eqnarray*}  
Clearly, $\sgn(f_1(\sigma))=\sgn(\sigma)$ and $\sgn(f_2(\sigma)) =-\sgn(\sigma)$. 
We also remark that $(\sigma T)_{j+2}= f_1(\sigma) (T_j)$ for $\sigma \in G_1$, $(\sigma T)_{j+1}= f_2(\sigma) (T_j)$ for $\sigma \in G_2$.  
For simplicity, set $\tau:=f_k(\sigma)$ for $\sigma \in G_k$. 
By these bijections, we see that the  $x_{a_j}$-part of  the left side of \eqref{quasi Garnir} for $j \ge 2$ is 
\begin{eqnarray*}
&&\Bigl( \sum_{\sigma \in G_1}(-1)^{j+2-1}\sgn(\sigma)e((\sigma T)_{j+2})  +\sum_{\sigma\in G_2}(-1)^{j+1-1}\sgn(\sigma)e((\sigma T)_{j+1}) \\
&& \qquad \qquad \qquad \qquad \qquad \qquad  \qquad \qquad +\sum_{\sigma\in G_3}(-1)^{j-1}\sgn(\sigma)e((\sigma T)_{j})  \Bigr) \otimes x_{a_j} \\
&=& (-1)^{j-1}\Bigl( \sum_{\tau\in G'_1}\sgn(\tau)e(\tau (T_j))  +\sum_{\tau\in G'_2}\sgn(\tau)e(\tau (T_j))\\
&& \qquad \qquad \qquad  \qquad \qquad \qquad \qquad \qquad  +\sum_{\sigma \in G'_3=G_3}\sgn(\sigma)e(\sigma (T_j)) \Bigr) \otimes  x_{a_j} \\
&=&(-1)^{j-1}g_{A_j,B}e(T_j) \otimes x_{a_j}\\
&=&0.
\end{eqnarray*}

Next we check the $x_{a_1}$-part. We decompose the $x_{a_1}$-part of the left side of \eqref{quasi Garnir} as follows
\begin{equation}\label{a_1 part}
\left( \sum_{\substack{\sigma\in S(A,B)\\ \sigma(b_2)=a_{i+1} }}\sgn(\sigma)e((\sigma T)_{1}) + \sum_{\substack{\sigma\in S(A,B)\\ \sigma(a_{i+1})=a_{i+1} }}\sgn(\sigma)e((\sigma T)_{1}) \right )\otimes x_{a_1}. 
\end{equation}

First, we consider the former, that is, the case $\sigma(b_2)=a_{i+1}$.
Set $\widetilde{A}=(A\cup \{b_2\})\backslash \{a_{i+1}\}$, $\widetilde{B}=\{b_1\}$, and
$\widetilde{T}:=(a_{i+1} \, \, a_i\, \, a_{i-1}\, \, a_{i-2}\, \, \cdots a_3\, \, a_2\, \, b_2)(T_1).$
Note that 
\begin{equation*}
\widetilde{T}= 
\ytableausetup
{mathmode, boxsize=2em}
\begin{ytableau}
b_2 & b_1 & c_1& c_2 & \cdots & a_1   \\
a_2 & a_{i+1}    \\
\vdots \\
a_{i}
\end{ytableau}.
\end{equation*}
We have  a bijection 
$$
\begin{array}{ccc}
\widetilde{f}:\{\sigma\in S(A,B) \mid \sigma(b_2)=a_{i+1}\}   & \longrightarrow &S_{\widetilde{T}}(\widetilde{A},\widetilde{B})                    \\
\rotatebox{90}{$\in$} & & \rotatebox{90}{$\in$} \\
\sigma                   & \longmapsto     & (a_{i+1}  \; \sigma(a_2)\; \sigma(a_3) \; \cdots \; \sigma(a_{i+1}))\sigma
\end{array}
$$
and can easily check that $(\sigma T)_1=\widetilde{f}(\sigma) \widetilde{T}$. For simplicity, set $\tau:=f(\sigma)$. 
Then  we have 
$$\sum_{\substack{\sigma\in S(A,B)\\ \sigma(b_2)=a_{i+1} }}\sgn(\sigma)e((\sigma T)_{1}) 
=(-1)^i\sum_{\substack{\tau\in S_{\widetilde{T}}(\widetilde{A},\widetilde{B}) }}\sgn(\tau)e(\tau \widetilde{T}) =(-1)^i g_{\widetilde{A},\widetilde{B}}e(\widetilde{T})=0. 
$$

Next, we consider the case $\sigma(a_{i+1})=a_{i+1}$ in \eqref{a_1 part}.
Set $\overline{A}=A\backslash \{a_{i+1}\}$ and
$$
\overline{T}:= 
\ytableausetup
{mathmode, boxsize=2em}
\begin{ytableau}
a_{i+1} & b_1 & c_1& c_2 & \cdots & a_1   \\
a_2 &  b_2    \\
a_3\\
\vdots \\
a_{i}
\end{ytableau}.
$$
We have
$$\{ \sigma \in S(A,B) \mid \sigma(a_{i+1})=a_{i+1}\}=S_{\overline{T}}(\overline{A},B),$$
and we have 
$$(\sigma T)_1 = (\sigma(a_2) \ \sigma(a_3) \ \cdots \sigma(a_i) \ a_{i+1}) \, \sigma \overline{T}$$
for $\sigma \in S(A,B)$ with $\sigma(a_{i+1})=a_{i+1}$. Hence $e((\sigma T)_1)=(-1)^{i-1} e( \sigma \overline{T})$ and 
\begin{eqnarray*}
\sum_{\substack{\sigma\in S(A,B)\\ \sigma(a_{i+1})=a_{i+1} }}\sgn(\sigma)e((\sigma T)_{1}) 
&=&(-1)^{i-1}\sum_{\sigma \in S_{\overline{T}}(\overline{A},B)}\sgn(\sigma)e(\sigma  \overline{T}) \\
&=&(-1)^{i-1} g_{\overline{A},B}e(\overline{T})\\
&=&0.
\end{eqnarray*}

It remains to check the $x_{b_1}$-part and the $x_{b_2}$-part of  the left side of \eqref{quasi Garnir}.  
Set $A':=A\backslash \{a_2\}$, $B':=\{b_2,a_2\}$ and $T':= ((b_1\, \, b_2\, \, a_2)T)_2$. Note that  
\begin{equation*}
T'= 
\ytableausetup
{mathmode, boxsize=2em}
\begin{ytableau}
a_1 & b_2 & c_1& c_2 & \cdots & b_1   \\
a_3 & a_2    \\
\vdots \\
a_{i+1}
\end{ytableau}.
\end{equation*}
We have  a bijection 
 $$f':\{\sigma\in S(A,B)\mid  \sigma(a_2)=b_1 \}\ni \sigma \longmapsto (\sigma(b_2)\, \, \sigma(b_1)\, \, b_1)\sigma \in S_{T'}(A',B'),$$ and then $(\sigma T)_2 = f'(\sigma)(T')$. 
Hence the $x_{b_1}$-part of the left side of \eqref{quasi Garnir} is 
\begin{eqnarray*}
\left(-\sum_{\substack{\sigma\in S(A,B)\\ \sigma(a_2)=b_1}}\sgn(\sigma)e((\sigma T)_{2})  ) \right) \otimes x_{b_1}
=-g_{A',B'}e(T')\otimes x_{b_1}=0.
\end{eqnarray*}

Set $A'':=A\backslash \{a_2\}$, $B'':=\{b_1,a_2\}$ and $T'':= ((b_2\, \, a_2)T)_2$. Note that 
\begin{equation*}
T''= 
\ytableausetup
{mathmode, boxsize=2em}
\begin{ytableau}
a_1 & b_1 & c_1& c_2 & \cdots & b_2   \\
a_3 & a_2    \\
\vdots \\
a_{i+1}
\end{ytableau}.
\end{equation*}
Set
\begin{eqnarray*}
L_1&=& \{\sigma \in S(A,B) \, \mid \, \sigma(a_2)=b_2\},\\
L_2&=& \{\sigma \in S(A,B) \, \mid \, \sigma(a_3)=b_2\},
\end{eqnarray*}
and 
\begin{eqnarray*}
L'_1&=& \{\tau \in S_{T''}(A'',B'') \, \mid \, \tau(b_1)=b_1\},\\
L'_2&=& \{\tau \in S_{T''}(A'',B'') \, \mid \, \tau(a_3)=b_1\}.
\end{eqnarray*}
If $\sigma\in L_1$ (resp. $\sigma\in L_2$), then $\sigma(b_1)=b_1$ (resp. $\sigma(a_2)=b_1$).
We also have
$$\{\sigma \in S(A,B) \mid \sigma(b_1), \sigma(b_2)\neq b_2\}=L_1\sqcup L_2$$
and 
$$S_{T''}(A'',B'')=L'_1\sqcup L'_2.$$

We have bijections 
$$f_1'':L_1 \ni \sigma \longmapsto (\sigma(b_2)\, \, b_2)\sigma \in L'_1$$ 
and  
$$f_2'':L_2 \ni \sigma \longmapsto (\sigma(b_2)\, \, b_2\, \, b_1)\sigma \in L'_2.$$
Note that $(\sigma T)_2 = f''_1(\sigma)(T'')$ for $\sigma \in L_1$, and $(\sigma T)_3 = f_2''(\sigma)(T'')$ for $\sigma \in L_2$. As before, set $\tau:= f_k''(\sigma)$ for $\sigma \in L_k$.   
Then the $x_{b_2}$-part of the left side of \eqref{quasi Garnir} is
\begin{eqnarray*}
&&\left(\sum_{\substack{\sigma\in L_1}}(-1)^{2-1}\sgn(\sigma)e((\sigma T)_{2})  +\sum_{\substack{\sigma\in L_2}}(-1)^{3-1}\sgn(\sigma)e((\sigma T)_{3}) \right) \otimes x_{b_2}\\
&=&\left(\sum_{\substack{\tau\in L'_1}}\sgn(\tau)e(\tau T'')  +\sum_{\substack{\tau\in L'_2 }}\sgn(\tau)e(\tau T'') \right) \otimes x_{b_2}\\
&=&g_{A'',B''}e(T'') \otimes x_{b_2}.
\end{eqnarray*}

So we are done.
\end{proof}

\begin{ex}\label{sigma-tau}
(1) Here we will check \eqref{tau-sigma} step by step.  
For simplicity, set $j=2$ (so $\sigma(a_3)=a_2$ now). Then we have 
$$
\tau(k)=\begin{cases}
a_2 & \text{if $k=a_2$,} \\
\sigma(a_2)  & \text{if $k=a_3$,} \\
\sigma(k) & \text{otherwise,}
\end{cases}
$$
$$T_2= 
\ytableausetup
{mathmode, boxsize=1.8em}
\begin{ytableau}
a_1 & b_1 &  c_1 & \cdots & a_2   \\
a_3 & b_2   \\
a_4\\
\vdots \\
\end{ytableau}
$$
and
$$
\tau(T_2)= \ytableausetup
{mathmode, boxsize=2.7em}
\begin{ytableau}
\tau(a_1) & \tau(b_1) &  \tau(c_1) & \cdots & \tau(a_2)   \\
\tau(a_3) & \tau(b_2)   \\
\tau(a_4)\\
\vdots \\
\end{ytableau}
= \ytableausetup
{mathmode, boxsize=2.7em}
\begin{ytableau}
\sigma(a_1) & \sigma(b_1) &  c_1 & \cdots & a_2   \\
\sigma(a_2) & b_2   \\
\sigma(a_4)\\
\vdots \\
\end{ytableau}.
$$
Since 
$$
\sigma T = \ytableausetup
{mathmode, boxsize=2.7em}
\begin{ytableau}
\sigma(a_1) & \sigma(b_1) &  c_1 & \cdots \\
\sigma(a_2) & b_2   \\
a_2 \\
\sigma(a_4)\\
\vdots \\
\end{ytableau},
$$
we have $\tau(T_2)=(\sigma T)_3$. 


(2) For the tableau $T$ of \eqref{T example} in Example~\ref{Garnir example}, we will check that the $x_1$-part of  
$$\sum_{\sigma \in S(A,B)}\partial_2(e(T)\otimes 1)$$
is 0 for $A=\{4,5\}$ and $B=\{2,3\}$.  The $x_1$-part is 
$$
\Bigg(
e(\ytableausetup{mathmode, boxsize=1.3em}
\begin{ytableau}
4 &2 & 6& 1\\
5 & 3   \\
\end{ytableau})
- e(
\begin{ytableau}
3 &2 & 6& 1\\
5 & 4    \\
\end{ytableau})
+
e(\begin{ytableau}
3 &2 & 6& 1\\
4 & 5    \\
\end{ytableau})
-e(
\begin{ytableau}
2 &3 & 6 & 1\\
4 & 5    \\
\end{ytableau})$$
$$\qquad \qquad \qquad \qquad+e(
\begin{ytableau}
2 &3 & 6 & 1\\
5& 4    \\
\end{ytableau})
+e(
\begin{ytableau}
2 &4 & 6& 1\\
3 & 5    \\
\end{ytableau})  \Biggr)\otimes x_1,
$$
but we have 
$$
e(\begin{ytableau}
3 &2 & 6 & 1\\
4 & 5    \\
\end{ytableau})
-e(
\begin{ytableau}
2 &3 & 6 & 1\\
4 & 5    \\
\end{ytableau})
+e(
\begin{ytableau}
2 &4 & 6 & 1\\
3 & 5    \\
\end{ytableau})=0
$$
and
$$
e(
\begin{ytableau}
4 &2 & 6 & 1\\
5 & 3   \\
\end{ytableau})
-e(
\begin{ytableau}
3 &2 & 6 & 1\\
5 & 4    \\
\end{ytableau})
+e(
\begin{ytableau}
2 &3 & 6 & 1\\
5& 4    \\
\end{ytableau})=0. 
$$
In fact, we can get the former (resp. latter) applying the Garnir element $g_{A,B}$ for $A=\{3,4\},B=\{2\}$ 
 (resp. $A=\{4\},B=\{2,3\}$) to 
 $$
 \begin{ytableau}
3 &2 & 6 & 1\\
4 & 5    \\
\end{ytableau}
 \qquad \qquad \Big(\, \text{resp.} \ 
 \begin{ytableau}
 5 &2 & 6 & 1\\
4 & 3   \\
\end{ytableau} \, \Bigr).$$
Recall that 
$$e(
\begin{ytableau}
4 &2 & 6 & 1\\
5 & 3   \\
\end{ytableau})
=- e(
\begin{ytableau}
 5 &2 & 6 & 1\\
4 & 3   \\
\end{ytableau} 
),
$$
and the same is true for the related tableaux. 
\end{ex}


\noindent{\it The proof of Theorem~\ref{1st thm}.}
First, we will show that $\cF_\bullet^{(n-2,2)}$ is a chain complex. 
For the tableau $T$ of \eqref{T for Gor} and any permutation $\sigma$ of $\{c_1, c_2, \ldots \}$, we have $e(T)= e(\sigma(T))$. Hence it is easy to see that $\partial_{i-1} \partial_i=0$ holds for $3 \le i \le n-3$. For 
\begin{equation}\label{3 rows}
T= 
\ytableausetup
{mathmode, boxsize=1.5em}
\begin{ytableau}
a_1 & b_1 & c_1& c_2 & \cdots    \\
a_2 & b_2    \\
a_3
\end{ytableau}
\end{equation}
we have 
$$
T_1= 
\ytableausetup
{mathmode, boxsize=1.5em}
\begin{ytableau}
a_2 & b_1 & c_1& c_2 & \cdots    \\
a_3 & b_2    
\end{ytableau},
\qquad 
T_2= 
\ytableausetup
{mathmode, boxsize=1.5em}
\begin{ytableau}
a_1 & b_1 & c_1& c_2 & \cdots    \\
a_3 & b_2    \\
\end{ytableau}, 
\qquad 
T_3= 
\ytableausetup
{mathmode, boxsize=1.5em}
\begin{ytableau}
a_1 & b_1 & c_1& c_2 & \cdots    \\
a_2 & b_2    \\
\end{ytableau}
$$
 and 
\begin{eqnarray*}
\partial_1 \partial_2(e(T) \otimes 1) &=& \partial_1(e(T_1) \otimes x_{a_1}-e(T_2) \otimes x_{a_2} + e(T_3) \otimes x_{a_3})  \\
&=& x_{a_1} f_{T_1} -x_{a_2}f_{T_2}+x_{a_3}f_{T_3}    \\
&=& (x_{a_1}(x_{a_2}-x_{a_3})-x_{a_2}(x_{a_1}-x_{a_3})+x_{a_3}(x_{a_1}-x_{a_2}))(x_{b_1}-x_{b_2})   \\
&=& 0,\nonumber
\end{eqnarray*}
so $\partial_1 \partial_2=0$. Finally, we will show that $\partial_{n-3} \partial_{n-2}=0$. 
Let $T \in \Tab(1^n)$ be the tableau in \eqref{T for Gor last}. Then $\partial_{n-3} \partial_{n-2}(e(T) \otimes 1)$ equals 
$$
\sum_{1 \le j < k < l \le n}
(-1)^{j+k+l}\biggl(
e \Bigl ( \, 
\ytableausetup{mathmode, boxsize=1.3em}
\begin{ytableau}
\vdots & j & l\\
\vdots & k\\
\vdots \\
\vdots 
\end{ytableau} \, \Bigr )
- 
e  \Bigl(\, 
\begin{ytableau}
\vdots & j & k\\
\vdots & l\\
\vdots \\
\vdots 
\end{ytableau} \,  \Bigr )
+ 
e  \Bigl(\, 
\begin{ytableau}
\vdots & k & j\\
\vdots & l\\
\vdots \\
\vdots 
\end{ytableau} \,  \Bigr ) 
\biggr ) 
\otimes x_{j}x_{k} x_{l}, 
$$
where all of the first columns of the above three tableaux are the ``transpose'' of $$\ytableausetup
{boxsize=2.5em}\begin{ytableau}
1 & 2 & \cdots & j-1 & j+1 & \cdots  & k-1 & k+1 & \cdots  & l-1 & l+1 &\cdots & n \end{ytableau}.$$
However,  we have 
$$
\ytableausetup
{boxsize=1.3em}
e \Bigl ( \, \begin{ytableau}
\vdots & j & l\\
\vdots & k\\
\vdots \\
\vdots 
\end{ytableau} \, \Bigr )
- 
e  \Bigl(\, 
\begin{ytableau}
\vdots & j & k\\
\vdots & l\\
\vdots \\
\vdots 
\end{ytableau} \,  \Bigr )
+ 
e  \Bigl(\, 
\begin{ytableau}
\vdots & k & j\\
\vdots & l\\
\vdots \\
\vdots 
\end{ytableau} \,  \Bigr )
= 0$$
by the Garnir relation. 
Hence $\partial_{n-3} \partial_{n-2}(e(T) \otimes 1)=0$, and $\partial_{n-3} \partial_{n-2}=0$. 
So we have shown that $\cF_\bullet^{(n-2,2)}$ is a chain complex.  Since $\Image \partial_1=\ISp_{(n-2,2)}$, $\cF_\bullet^{(n-2,2)}$ is a subcomplex of  
a minimal free resolution of $R/\ISp_{(n-2,2)}$.

Recall that we regard $F_i$ as an $\fS_n$-module by $\sigma (v \otimes f) := \sigma v \otimes \sigma f \in  V_\lambda \otimes_K R(-j) =F_i$, where $\lambda$ is a suitable partition of  $n$ and 
$j$ is a suitable integer uniquely determined by $i$. In the previous section, we have shown that  $\partial_i :F_i \to F_{i-1}$ is an $\fS_n$-homomorphism.  The  restriction 
$$[\partial_i]_j : [F_i]_j =V_\lambda \otimes_K [R(-j)]_j = V_\lambda \otimes_K K \too V_{\lambda'} \otimes_K R_l = [F_{i-1}]_j$$ is also an $\fS_n$-homomorphism, where $l=1$ if $2 \le i \le n-3$, and $l=2$ if $i=1, n-2$. 
Since $V_\lambda \otimes_K K \cong V_\lambda$ is  irreducible as an $\fS_n$-module and $[\partial_i]_j$ is clearly nonzero, we have  $[\partial_i]_j$ is injective. Since 
$\mu(\Ker \partial_{i-1})=\beta_{i,j}(R/\ISp_{(n-2,2)}) =\dim V_\lambda = \dim_K [\Image  \partial_i]_j $ 
for $i \ge 2$ by Lemma~\ref{Gor Betti}, $\cF_\bullet^{(n-2,2)}$ is a (minimal) free resolution of $R/\ISp_{(n-2,2)}$.  
Here $\mu(-)$ denote the minimal number of generators as an $R$-module. 
\qed

\section{The case $(d,d,1)$: Construction}
For $R/\ISp_{(d,d,1)}$, we define the chain complex 
\begin{equation}\label{linear}
\cF_\bullet^{(d,d,1)}:0 \too F_d \stackrel{\partial_d}{\too}  F_{d-1} \stackrel{\partial_{d-1}}{\too}  \cdots \stackrel{\partial_2}{\too}  F_1 \stackrel{\partial_1}{\too} F_0 \too 0
\end{equation}
of graded free $R$-modules as follows. 
Here $F_0 =R$  and 
$$
F_i = V_{(d, d-i+1, 1^i)} \otimes_K R(-d-i-1)
$$
for $1 \le i \le d$. As before,  set $\partial_1(e(T)\otimes 1) :=f_T \in R =F_0$. To describe $\partial_i$ for $i \ge 2$, we need the preparation. For 

\begin{equation}\label{T for linear}
T= 
\ytableausetup
{mathmode, boxsize=3.1em}
\begin{ytableau}
a_1 & b_2 &  \cdots & b_{d-i+1}& b_{d-i+2}& \cdots & b_d   \\
a_2 & c_2 &  \cdots & c_{d-i+1}  \\
\vdots \\
a_{i+2}
\end{ytableau}
\end{equation}
in $\Tab(d, d-i+1, 1^i)$ and $j$ with $1 \le j \le i+2$, set 
$$
T_j  :=
\ytableausetup
{mathmode, boxsize=3.1em}
\begin{ytableau}
a_1 & b_2 &  \cdots & b_{d-i+1} &  b_{d-i+2} &   b_{d-i+3} & \cdots & b_d   \\
a_2 & c_2 &  \cdots & c_{d-i+1} & a_j \\
\vdots \\
a_{j-1}\\
a_{j+1} \\
\vdots\\
a_{i+2}
\end{ytableau}
$$ 
in $\Tab(d, d-i+2, 1^{i-1})$. 
Then we set 
\begin{equation}\label{dif for linear}
\partial_i(e(T) \otimes 1) := \sum_{j=1}^{i+2} \sum_{\sigma \in H}  (-1)^{j-1} e(\sigma(T_j)) \otimes x_{a_j} \in V_{(d, d-i+2, 1^{i-1})} \otimes_K R(-d-i) =F_{i-1}
\end{equation}
for $3 \le i \le d-1$, where $H$ is the set of permutations of $\{  b_{d-i+2}, b_{d-i+3}, \ldots  ,b_d \}$ satisfying $\sigma(b_{d-i+3}) <  \sigma(b_{d-i+4}) <  \cdots < \sigma(b_d)$, and 
$$\partial_2(e(T) \otimes 1) = \sum_{j=1}^3  (-1)^{j-1} e({T_j}) \otimes x_{a_j} \in V_{(d, d, 1)} \otimes_K R(-d-2) =F_1$$ 
for $T \in \Tab(d, d-1, 1,1)$. That these $\partial_i$ are well-defined will be shown in 
Theorem~\ref{wd linear}. We are not sure whether $\partial_d$ can be defined in the same way, that is, the well-definedness is not clear in this case. However, for our purpose (i.e., to show $\partial_{d-1} \circ \partial_d=0$), it suffices to define $\partial_d(e(T) \otimes 1)$ for $T \in \SYT(d, 1^{d+1})$. Since $\{ e(T) \mid T \in \SYT(d, 1^{d+1})\}$ is a basis, we do not have to care the well-definedness. Anyway, we define $\partial_d(e(T) \otimes 1)$ for $T \in \SYT(d, 1^{d+1})$ by \eqref{dif for linear}.

\begin{ex}\label{(4,4,1)}
Our minimal free resolution $\cF_\bullet^{(4,4,1)}$ of $R/\ISp_{(4,4,1)}$ is of the form  
$$
0 \too V_{\ytableausetup{boxsize=0.25em} 
\begin{ytableau}
{} & {} & {} & {}\\
\\
\\
\\
\\
\\
\end{ytableau}}\otimes_K R(-9) \stackrel{\partial_{4}}{\too} 
 V_{\ytableausetup{boxsize=0.25em} 
\begin{ytableau}
{} & {} & {} & \\
{}  & \\
\\
\\
\\
\end{ytableau}} \otimes_K R(-8) \stackrel{\partial_{3}}{\too} 
V_{\ytableausetup{boxsize=0.25em} 
\begin{ytableau}
{} & {} & {} & \\
{}  & & \\
\\
\\
\end{ytableau}} \otimes_K R(-7)  \qquad \qquad  \qquad \quad 
$$
$$
\qquad \qquad \qquad \qquad \qquad  \qquad \qquad  
\stackrel{\partial_{2}}{\too} 
V_{\ytableausetup{boxsize=0.25em} 
\begin{ytableau}
{} & {} & {}  & \\
 {} & {} & & \\
\\
\end{ytableau}} \otimes_K R(-6) \stackrel{\partial_{1}}{\too} 
R \too 0. 
$$
The differential maps are given by  
\begin{eqnarray*}
\partial_4 \bigl ( e \bigl( \, 
\ytableausetup
{mathmode, boxsize=1em}
\begin{ytableau}
1 & 2 & 3 & 4\\
5   \\
6 \\
7 \\
8 \\
9
\end{ytableau}
\, \bigr)  \otimes 1 \bigr) 
&=& \bigl ( \, 
e \bigl( \, 
\ytableausetup
{mathmode, boxsize=1em}
\begin{ytableau}
5 & 2 & 3 & 4\\
6  & 1 \\
7 \\
8 \\
9 
\end{ytableau}
\, \bigr)   + 
e \bigl( \, 
\ytableausetup
{mathmode, boxsize=1em}
\begin{ytableau}
5 & 3 & 2 & 4\\
6  & 1 \\
7 \\
8 \\
9 
\end{ytableau}
\, \bigr) + 
e \bigl ( \, 
\ytableausetup
{mathmode, boxsize=1em}
\begin{ytableau}
5 & 4 & 2 & 3\\
6  & 1 \\
7 \\
8 \\
9 
\end{ytableau}
\, \bigr)  \, \bigr )
\otimes x_1 \\
& & -  
 \bigl ( \ 
e \bigl ( \, 
\ytableausetup
{mathmode, boxsize=1em}
\begin{ytableau}
1 & 2 & 3 & 4\\
6  & 5 \\
7 \\
8 \\
9 
\end{ytableau}
\, \bigr )   + 
e \bigl ( \, 
\ytableausetup
{mathmode, boxsize=1em}
\begin{ytableau}
1 & 3 & 2 & 4\\
6  & 5 \\
7 \\
8 \\
9 
\end{ytableau}
\, \bigr) + 
e \bigl ( \, 
\ytableausetup
{mathmode, boxsize=1em}
\begin{ytableau}
1 & 4 & 2 & 3\\
6  & 5 \\
7 \\
8 \\
9 
\end{ytableau}
\, \bigr)  \,  \bigr)
\otimes x_5
 \\
& & \qquad  \vdots \\
& & \qquad  \vdots \\
 & & -  ( \ 
e \bigl ( \, 
\ytableausetup
{mathmode, boxsize=1em}
\begin{ytableau}
1 & 2 & 3 & 4\\
5  & 9 \\
6 \\
7 \\
8 
\end{ytableau}
\, \bigr )   + 
e \bigl ( \, 
\ytableausetup
{mathmode, boxsize=1em}
\begin{ytableau}
1 & 3 & 2 & 4\\
5  & 9 \\
6 \\
7 \\
8 
\end{ytableau}
\, \bigr) + 
e \bigl( \, 
\ytableausetup
{mathmode, boxsize=1em}
\begin{ytableau}
1 & 4 & 2 & 3\\
5  & 9 \\
6 \\
7 \\
8 
\end{ytableau}
\, \bigr)  \, \bigr )
\otimes x_9
\end{eqnarray*}
and 
\begin{eqnarray*}
\partial_2 \bigl( e\bigl ( \, 
\ytableausetup
{mathmode, boxsize=1em}
\begin{ytableau}
1 & 2 & 3 & 4\\
5  & 6 & 7\\
8 \\
9
\end{ytableau}
\, \bigr )  \otimes 1 \bigr ) 
&=& e \bigl( \, 
\ytableausetup
{mathmode, boxsize=1em}
\begin{ytableau}
5 & 2 & 3 & 4\\
8  & 6 & 7 & 1\\
9 \\
\end{ytableau}
\, \bigr )  \otimes x_1 
- e \bigl( \, 
\ytableausetup
{mathmode, boxsize=1em}
\begin{ytableau}
1 & 2 & 3 & 4\\
8  & 6 & 7 & 5\\
9 \\
\end{ytableau}
\, \bigr  )  \otimes x_5\\
& & 
+ e\bigl ( \, 
\ytableausetup
{mathmode, boxsize=1em}
\begin{ytableau}
1 & 2 & 3 & 4\\
5  & 6 & 7 & 8\\
9 \\
\end{ytableau}
\, \bigr )  \otimes x_8 
- e \bigl( \, 
\ytableausetup
{mathmode, boxsize=1em}
\begin{ytableau}
1 & 2 & 3 & 4\\
5  & 6 & 7 & 9\\
8 \\
\end{ytableau}
\,\bigr  )  \otimes x_9.
\end{eqnarray*}
\end{ex}

\begin{thm}\label{2nd thm}
If $\chara(K)=0$, the complex $\cF_\bullet^{(d,d,1)}$ of \eqref{linear} is a minimal free resolution  of $R/\ISp_{(d,d,1)}$.  
\end{thm}


\section{The case $(d,d,1)$: Proof}

\begin{lem}\label{linear Betti}
For all $i$ with $1\leq i \leq d$, we have 
$$\beta_{i,i+d+1}(R/\ISp_{(d,d,1)})= \dim_K V_{(d,d-i+1,1^{i})}. $$
\end{lem}

\begin{proof}
By the hook formula, we have 
\begin{eqnarray*}
\dim_K V_{(d,d-i+1,1^{i})}&=&  \frac{(2d+1)!}{(d+i+1)\frac{d!}{i}(d+1)(d-i)!i!}\\
&=&\frac{(2d+1)!}{(d+i+1)(d+1)!(d-i)!(i-1)!}
\end{eqnarray*}
for all $i$ with $1\leq i \leq d$.

Since $\ISp_{(d,d,1)}$ is a Cohen-Macaulay ideal of codimensin $d$ and has  a $(d+2)$-linear resolution, 
we have 
$$\beta_i(R/\ISp_{(d,d,1)})=\beta_{i,i+d+1}(R/\ISp_{(d,d,1)})$$
for all $i \ge 1$, and  \cite[Theorem~3.5.17]{V} implies that
\begin{eqnarray*}
\beta_i(R/\ISp_{(d,d,1)})&=& \prod_{j=1}^{i-1}\frac{d+1+j}{i-j} \prod_{j=i+1}^{d}\frac{d+1+j}{j-i} \\
&=&\frac{(d+i)!}{(i-1)!(d+1)!} \cdot \frac{(2d+1)!}{(d-i)!(d+i+1)!}  \\
&=&\frac{(2d+1)!}{(d+i+1)(d+1)!(d-i)!(i-1)!}.
\end{eqnarray*}
So we are done.
\end{proof}

\begin{thm}\label{wd linear}
The maps $\partial_i$ defined in the previous section are well-defined.
\end{thm}

\begin{proof}
The well-definedness of $\partial_1$ is nothing other than that of \eqref{Isom}, and that of $\partial_d$ is clear by the construction. So we may assume that 
 $2 \leq i \leq d-1$. By Proposition~\ref{relation generated},
if suffices to show that 
\begin{equation}\label{quasi Garnir2}
\sum_{\sigma\in S_T(A,B)} \sgn(\sigma) \,\partial_i(e(\sigma T)\otimes 1)
\end{equation} 
equals 0 for  the tableau $T$ of \eqref{T for linear}. 
The cases $A=\{b_k, c_k\}$ and $B=\{b_{k+1}\}$ for $2 \le k \le d-i$ are easy. 
The non-trivial cases are 
\begin{itemize}
\item[(1)]  $A=\{a_1,\ldots,a_{i+2}\}, B=\{b_2 \}$, 
 \item[(2)] $A=\{a_2,\ldots,a_{i+2}\}, B=\{b_2, c_2\}$,
 \item[(3)]  $A=\{b_{d-i+1}, c_{d-i+1}\}$, $B=\{b_{d-i+2}\}$. 
\end{itemize}

First, we treat the case (1). 
We may assume that $b_2<a_1<a_2<\cdots<a_{i+2}$. 
Fix  $j$ with $1\leq j \leq i+2$, and set $A_j:=A\backslash \{a_j\}$.
Since $\tau \in H$ is a permutations of $\{  b_{d-i+2}, b_{d-i+3}, \ldots  ,b_d \}$, $\sigma\in S(A_j,B)$ and $\tau \in H$ are disjoint. Hence  we have  
\begin{eqnarray*}
(\, \text{the $x_{a_j}$-part of  \eqref{quasi Garnir2}} \, )=(-1)^{j-1}\sum_{\tau \in H} g_{A_j,B} e(\tau(T_j)) \otimes x_{a_j}.
\end{eqnarray*}
For each $\tau$, we can show that $g_{A_j,B} e(\tau(T_j)) =0$ by an argument similar to the proof of Theorem \ref{wdGor}. 
That $(\, \text{the $x_{b_2}$-part of    \eqref{quasi Garnir2}} \, )=0$, and the case (2) can be proved in a similar way. 

For the case (3),  the tableau 
$$
\ytableausetup
{mathmode, boxsize=4.3em}
\begin{ytableau}
a_1 & b_2 &  \cdots & \sigma(b_{d-i+1})& b_k & \cdots &\sigma(b_{d-i+2})& \cdots  \\
a_2 & c_2 &  \cdots & \sigma(c_{d-i+1}) & a_j \\
\vdots \\
\end{ytableau}
$$
appears in the $x_{a_j}$-part of \eqref{quasi Garnir2} for $d-i+3 \le k \le d$ (here we assume that $j \ge 3$ for notational simplicity). 
However the above tableau and 
$$
\ytableausetup
{mathmode, boxsize=4.3em}
\begin{ytableau}
a_1 & b_2 &  \cdots & b_k &\sigma(b_{d-i+1})&\sigma(b_{d-i+2})& \cdots  \\
a_2 & c_2 &  \cdots & a_j & \sigma(c_{d-i+1}) \\
\vdots \\
\end{ytableau}
$$
give the same polytabloid $e(-)$, so the previous argument also works.  
\end{proof}

\medskip

\noindent{\it The proof of Theorem~\ref{2nd thm}.}
First, we will show that $\cF^{(d,d,1)}_\bullet$ is a chain complex. 
It is easy to see that $\partial_{i-1}\partial_i =0$ for $3 \le i \le d$. To show $\partial_1\partial_2 =0$,  consider 
$$
T= 
\ytableausetup
{mathmode, boxsize=2em}
\begin{ytableau}
a_1 & b_2 &  \cdots & b_{d-1} & b_d   \\
a_2 & c_2 &  \cdots & c_{d-1}  \\
a_3\\
a_4
\end{ytableau}.
$$
Then we have 
$$T_1= 
\ytableausetup
{mathmode, boxsize=2em}
\begin{ytableau}
a_2 & b_2 &  \cdots & b_{d-1} & b_d   \\
a_3 & c_2 &  \cdots & c_{d-1}  & a_1\\
a_4
\end{ytableau},
\qquad 
T_2= 
\ytableausetup
{mathmode, boxsize=2em}
\begin{ytableau}
a_1 & b_2 &  \cdots & b_{d-1} & b_d   \\
a_3 & c_2 &  \cdots & c_{d-1}  & a_2\\
a_4
\end{ytableau}
$$
$$T_3= 
\ytableausetup
{mathmode, boxsize=2em}
\begin{ytableau}
a_1 & b_2 &  \cdots & b_{d-1} & b_d   \\
a_2 & c_2 &  \cdots & c_{d-1}  & a_3\\
a_4
\end{ytableau},
\qquad 
T_4= 
\ytableausetup
{mathmode, boxsize=2em}
\begin{ytableau}
a_1 & b_2 &  \cdots & b_{d-1} & b_d   \\
a_2 & c_2 &  \cdots & c_{d-1}  & a_4\\
a_3
\end{ytableau}
$$
in the notation of the previous section, and it holds that  
$$
\partial_1 \partial_2 (e(T) \otimes 1)= x_{a_1}f_{T_1}-x_{a_2}f_{T_2}+x_{a_3}f_{T_3}-x_{a_4}f_{T_4}.
$$
One can show that this equals 0  by a direct computation (note that each part of the right side can be divided by $\prod_{i=2}^{d-1}(x_{b_i} -x_{c_i})$). 

\qed

\section*{Acknowledgements} 
The authors  are grateful to Professor Satoshi Murai for stimulating discussion in the beginning of this project.  They also thank Professors Steven V. Sam, Junzo Watanabe, 
Alexander Woo, and the anonymous referees of the previous version for valuable comments/information on \cite{BSS,LL} and related works.

\end{document}